\documentclass[a4paper,12pt]{article}  
\usepackage[utf8]{inputenc}  
\usepackage[T2A]{fontenc}  
\usepackage[english]{babel}  
\usepackage[top=20mm, left=20mm, right=20mm, bottom=20mm]{geometry}  
\usepackage[pagebackref]{hyperref} 
\usepackage{sectsty}

\subsectionfont{\normalsize}

\usepackage{amsmath} 
\usepackage{amssymb} 
\usepackage{amsthm} 
\usepackage{amsfonts}
\usepackage{mathrsfs} 
\usepackage{euscript} 
\usepackage{bm, enumitem}

\usepackage{tikz}
    \usetikzlibrary{arrows.meta}
    \usetikzlibrary{calc}
    
    \tikzset{gdst/.style=
    {circle, draw=black!50, very thick, minimum height=1.2cm, inner sep=2pt, text centered, }, }

\usepackage{float, graphicx}
\graphicspath{{pictures/}}
\usepackage{color}

\newcommand{\rr}{\mathbb{R}}
\newcommand{\nn}{\mathbb{N}}
\newcommand{\zz}{\mathbb{Z}}

\newcommand{\al}{\alpha}
\newcommand{\be}{\beta}
\newcommand {\da} {\delta}
\newcommand {\Da} {\Delta}

\newcommand {\Ga} {\Gamma}

\newcommand {\sa} {\sigma}
\newcommand {\Sa} {\Sigma}

\newcommand{\bd}{{\bm{d}}}

\newcommand{\bi}{{\bm{i}}}

\newcommand{\bma}{{\bm{\alpha}}}

\newcommand{\bmb}{{\bm{\beta}}}

\newcommand{\bmg}{{\bm{\gamma}}}

\newcommand{\IN}{{\subset}}
\newcommand{\NI}{{\supset}}

\newcommand{\pr}{\mathrm{pr}}

\newcommand{\mmm}{{\setminus}}
\newcommand{\dd}{{\partial}}
\newcommand{\8}{{\infty}}
\newcommand{\0}{{\varnothing}}
\renewcommand{\leq}{\leqslant}

\newcommand{\hT}{{\hat T}}
\newcommand{\eS}{{\EuScript S}}

\DeclareMathSymbol{\sm}{\mathbin}{AMSa}{"39}

\newtheorem{thm}{\bf Theorem}
\newtheorem{cor}[thm]{\bf Corollary}
\newtheorem{lem}[thm]{\bf Lemma}
\newtheorem{prop}[thm]{\bf Proposition}
\newtheorem{dfn}{\bf Definition}
\theoremstyle{definition}
\newtheorem{rmk}{Remark}

\newcommand{\beq}{\begin{equation}}
\newcommand{\eeq}{\end{equation}}

\usepackage[affil-it]{authblk}
\title{On  the intersection of fractal cubes}
\author{
    Tetenov Andrei\thanks{e-mail: 
    \href{mailto:atet@mail.ru}{\texttt{atet@mail.ru}}}
    \quad and \quad
    Drozdov Dmitry\thanks{e-mail: 
    \href{mailto:dimalek97@yandex.ru}{\texttt{dimalek97@yandex.ru}}}}
\affil{Sobolev Institute of Mathematics, Novosibirsk, Russia}
\begin{document}

\maketitle
\noindent\rule{\textwidth}{1pt}
\begin{abstract}
Let $n\ge 2$ and let $D=\{d_1,\cdots,d_N\}\subset\{0,1,\dots,n-1\}^k$.    The set $D$ and the integer $n$  determine a system of contractions $\mathcal S=\{S_j(x)=\frac{1}{n}(x+d_j)\}_{j=1}^N$ in $R^k$, whose attractor $K$ satisfies the set equation $nK=K+D$ and is called a fractal $k$-cube of order $n$.

We consider the intersections of fractal $k$-cubes of order $n$ in $\rr^k$ and intersections of their respective opposite $l$-faces, $0\le l<k$. The main result of the paper is the theorem on the representation of such an intersection as the attractor of a graph-directed system of similarities in terms of the sets of units corresponding to these cubes and intersections of pairs of $l$-faces. As a corollary, we prove the dimension formula for the intersection and the condition of finiteness of its measure. Another corollary gives the conditions under which the intersections have the given cardinality. Applying these techniques, we obtain the conditions under which a fractal $k$-cube has the finite intersection property and the conditions under which the fractal cube is a dendrite. 
\end{abstract}
\noindent\rule{\textwidth}{1pt}

Let $n\ge 2$ and let $D=\{d_1,\cdots,d_N\}\subset\{0,1,\dots,n-1\}^k$.  We call the set $D$  a digit set.  The set  $D$ and the integer $n$  determine a system of contraction similarities $\eS=\{S_j(x)=\frac{1}{n}(x+d_j)\}_{j=1}^N$ in $\rr^k$, whose attractor $K$ satisfies the set equation
\begin{equation}\label{fractalcube}
K=\dfrac{K+D}{n}
\end{equation}
We call $K$ a \emph{fractal k-cube} of the order n. In the special cases where $k=2$ and $k=1$, we call $K$ a \emph{fractal square} and a \emph{fractal segment}, respectively.

Each fractal $k$-cube $K$  is contained in the unit cube $P^k$. With a slight abuse of terminology, we will sometimes call isometric images of a fractal $k$-cube by the same name.

\section{Unit $k$-cube, its faces and sections}
Every fractal $k$-cube is a subset of the unit $k$-cube $P^k$ in $\rr^k$, so we introduce some notation related to the boundary, faces, and sections of the unit cube $P^k$.\\

\subsection{ The family of faces of the unit cube.} 
We denote by $A_k$ the set $\{-1,0,1\}^k$. There is a natural one-to-one correspondence between the family of all faces of the cube $P^k$ and the set $A_k$. To establish it, we consider the center $(1/2,\ldots,1/2)$ of the cube $P^k$ which we denote by $c$.
For each vector $\bma=(\al_1,\ldots\al_k)\in A_k$ we put in correspondence the unique face $P_\bma$ of the cube $P^k$, whose center is $c_\bma=c+\bma/2$, and conversely.  The opposite vector $-\bma$ defines the face $P_{-\bma}$ with the center $c_{-\bma}=c-\bma/2$, which is parallel and opposite to $P_{\bma}$, therefore $P_{-\bma}+\bma=P_\bma$.

 \begin{figure}[H]
     \centering
\includegraphics[width=0.42\textwidth]{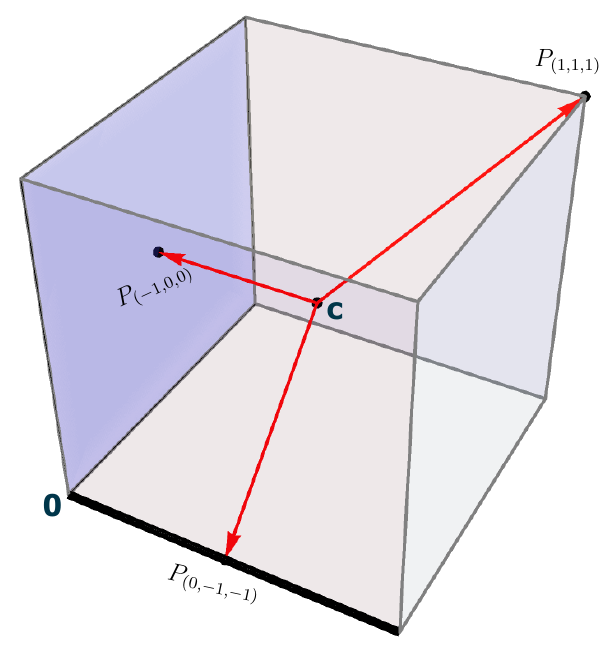}
   \hspace{0.05\textwidth}
 \includegraphics[width=0.47\textwidth]{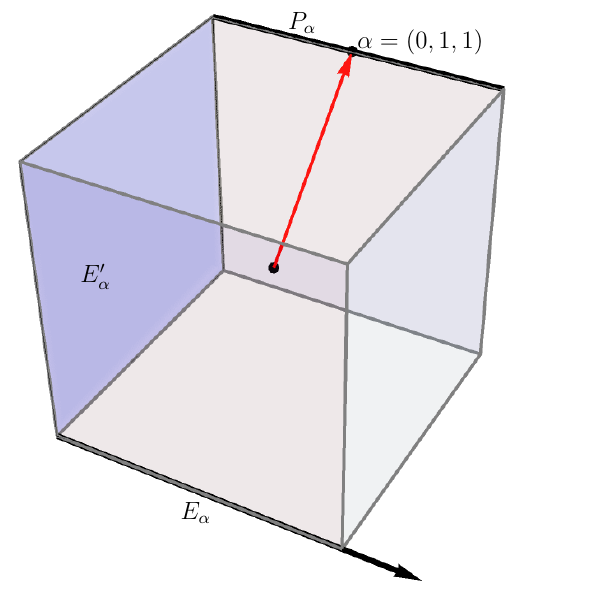}
      \caption{The faces $P_\bma$ of the unit cube.(left) and the subspaces $E_\bma$ and $E'_\bma$ (right)}
    
 \end{figure}

The equality $P_{-\bma}+\bma=P_\bma$ shows that the cubes $P^k$ and $P^k+\bma$ intersect on their common face.     This common face is the face $P_\bma$  of the cube $P^k$ and is the face $P_{-\bma}+\bma$ of the cube $P^k+\bma$.\\

\subsection{ The relations between faces and subspaces.} 
 For each $\bma\in A_k$  the sets $J_\bma=\{i:\al_i=0\}$  and $J'_\bma=\{i:|\al_i|= 1\}$  give rise to orthogonal complementary subspaces $E_\bma=\mathrm{span}\{e_i, i\in J_\bma\}$ and $E'_\bma=\mathrm{span}\{e_i, i\in J'_\bma\}$. It follows that $\bma\in E'_\bma$ and that the face
$P_\bma$ is parallel to the subspace $E_\bma$. We also define
$P'_\bma=P\cap E'_\bma$, which is the unique face, complementary to $P_\bma$ and containing the origin.
 We denote the orthogonal projections of $\rr^k$ to $E_\bma$ by $\pr_{\bma}$ and   to $E'_\bma$  by $\pr'_{\bma}$. \\
 
  We define $|\bma|=(|\al_1|+\ldots  +|\al_k|)$ and observe that the dimensions of the subspaces $E_\bma$ and $E '_\bma$ are equal to $j_\bma=k-|\bma|$ and $j_\bma'=|\bma|$, respectively.\\

   Notice that all faces $P_\bma$, for which the set $J_\bma$ is the same, are isometric and parallel. Since each nonzero $\al_i$ is equal to $+1$ or $-1$, there are $ 2^{|\bma|}$ different vectors $\bmb\in A_k$ for which $J_\bmb=J_\bma$. For any vector $\bmb\in A_k$ the equality $J_\bmb=J_\bma$ implies the equalities $E_\bmb=E_\bma$, $E'_\bmb=E'_\bma$, $P'_\bmb=P'_\bma$, as well as the respective equalities for the projections $\pr_\bmb,\pr'_\bmb$ and the dimensions $j_\bmb$ and $j'_\bmb$.  They all depend on $J_\bma$ only. 
 
 Whenever it is known which $\bma$  is referred to, we simply write $J,J',\pr,\pr',j$ and $j'$, omitting the subscript $\bma$.

\begin{dfn}\label{INperp}
Given $\bma,\bmb\in A_k$, we  write $\bma\sqsubseteq\bmb$ if for any $i\in J'_\bma$, 
 $\al_i=\be_i$ or equivalently, $J'_\bma\subseteq J'_\bmb$.  
 If $J'_\bma\cap J'_\bmb=\0$ we say $\bma$ and $\bmb$ are {\em complementary} and write $\bma\perp\bmb$.
 We denote by $A_\bma$ the set of all $\bmb\in A$, complementary to $\bma$.
 \end{dfn}
 
If $\bmb=\bma+\bmg$ is the sum of complementary non-null vectors, then $\bma\sqsubset\bmb$.
On the contrary, if $\bma\sqsubset\bmb$, then the vector $\bmg=\bmb-\bma$ is also the element of $A_k$, and hence
$\bmg\perp\bma$.

If $\bma\bot\bmb$, then the subspaces $E'_\bma$ and $E'_\bmb$ are orthogonal, and $E'_\bma\cap E'_\bmb =\{0\}$. Consequently, $E'_\bma\IN E_\bmb $ and $E'_\bmb\IN E_\bma$.

 If $\bma\bot\bmb$, then $P_{\bma+\bmb}=P_\bma\cap P_\bmb$.
 
A vector $\bma\in A_k$ is maximal with respect to the $\sqsubset$ relation, if and only if $J_\bma=\0$. In this case $P_\bma=\{c+\bma/2\}$ is a vertex of the cube $P^k$.

\subsection{ The faces containing the origin.}  
We denote by $P^0_{\bma}=P^k\cap E_\bma$ the face that is parallel to $P_\bma$ and contains the origin. 
The translation vector $\bma^0$ that sends $P^0_{\bma}$ to $P_\bma$  has the entries $\bma_i^0=\max\{\bma_i,0\}$ for any $i=1,\ldots  ,k$.
 The translation vector $(-\bma)^0$ sending $P^0_{\bma}$ to $P_{-\bma}$ is $(-\bma)^0=\bma^0-\bma.$ 
  Thus, $P_\bma=\bma^0 +P^0_{\bma}$ and $P_{-\bma}=\bma^0-\bma+P^0_{\bma}$.
  
 The complementary face $P'_\bma=E'_\bma\cap P$  contains the origin by its definition.

  If $\bma\bot\bmb$, then $P^0_{\bma+\bmb}=P^0_\bma\cap P^0_\bmb=P_\bma\cap P_\bmb-\bma^0-\bmb^0$.

  If $P_\bma$ is a vertex of the cube $P^k$ then $P^0_\bma=\{\bm{0}\}$ is the vertex at the origin.

   \subsection { The boundary of the unit cube $P^k$ and its faces.} 
Since $P_\bma=P^k\cap(\bma+P^k)$, the set
$\{\bma+P^k, \bma\in A\mmm \{0\}\}$ is the set of all neighbors of $P^k$ in the family $\{d+P^k, d\in \zz^k\}$, and the boundary of $P^k$ is represented by the formula \begin{equation}\label{dpk}
    \dd P^k=\bigcup\limits_{\bma\in A_k\mmm \{0\}}P_\bma
\end{equation}

 Similarly, for each $\bma\in A$  the boundary of $P_\bma$ is represented by the equation.
 
\begin{equation}\label{dpal}
    \dd P_\bma=\bigcup\limits_{\bmb\sqsupset\bma}P_{\bmb}=\bigcup\limits_{\bmg\in A_\bma\mmm\{0\}}P_{\bma+\bmg}
\end{equation}
Taking into account the equalities $P_\bma=\bma^0 +P^0_{\bma}$ and
$P_{\bma+\bmb}=P^0_{\bma+\bmb}+\bma^0+\bmb^0$, we arrive to a similar equation for $\dd P^0_\bma$:

\begin{equation}\label{dpa0}
    \dd P^0_\bma=\bigcup\limits_{\bmb\in A_\bma\mmm\{0\}}P^0_{\bma+\bmb}+\bmb^0.
\end{equation}
\\


\section{Projections, sections, faces and $p$-th refinement of a fractal $k$-cube $K$.}

In this section, we prove that a projection of a fractal cube $K$ to a face $P_\bma$ and the intersection $K_\bma=P_\bma\cap K$ are fractal cubes. 
\\

\subsection{The projections and $\al$-sections of a  fractal cube.}
Let $K\IN P^k$ be a fractal $k$-cube of order $n$ with a digit set $D\IN \{0,1,\ldots,n-1\}^k$.\\ 
Given $d_i\in \{0,1,\ldots,n-1\}^k$, we denote $S_i(x)=\dfrac{x+d_i}{n}$ and $S_{i_1\ldots  i_p}=S_{i_1}\cdot\ldots  \cdot S_{i_p}$.

 The fixed point of the map $S_i$ is $\dfrac{d_i}{n-1}$. It will be denoted by $c_i$.

\begin{lem}\label{lem1}
For any $\bma\in A_k$ projection $\pr'_{\bma}(K)$ is a fractal cube with a digit set  $\pr'_{\bma}(D)$. 
\end{lem}
\begin{proof} Obviously,
$\pr'_\bma(K)=\pr'_{\bma}\dfrac{K+D}{n}=\dfrac{\pr'_\bma(K) +\pr'_{\bma}(D)}{n}$.   
\end{proof}

\begin{prop}\label{c0inK}
Let $d_{i_1},\ldots  ,d_{i_p},d_0\in I^k$ and $c_0={d_0}/({n-1})$. Then
$c_0\in S_{i_1\ldots  i_p}(P)$ if and only if all $d_{i_j}$ are equal to $d_0$.  
 Moreover, the point $c_0$ lies in $K$ if and only if $d_0\in D$.
\end{prop}

\begin{proof}
For any $i=1,\ldots  ,k$ put $\al_i=-1$ if $d_{0,i}=0$, $\al_i=1$ if $d_{0,i}=n-1$ and $\al_i=0$ otherwise.
Let $\bma=(\al_1,\ldots  ,\al_k)$. Then $c_0$ is an interior point of the face $P_\bma$ and, for that reason, it is an interior point of a cube $S_0(P)$ relative to the cube $P$. Hence, for any $d_i\neq d_0$, $c_0\notin S_i(P)$.

Thus, if  we take  the map $S_{i_1i_2\ldots  i_p}$ for which $d_{i_1}\neq d_0$, then $c_0\notin S_{i_1}(P)$ 
and consequently, $c_0\notin S_{i_1\ldots  i_p}(P)$.

 Now  suppose
 for some $j\le p$, $d_{i_1}=\ldots  =d_{i_{j-1}}=d_0$ and $d_{i_j}\neq d_0$. The previous argument shows that $c_0=S_0^j(c_0)\notin S_0^{j-1} S_{i_j}(P)=S_{i_1\ldots  i_j}(P)$ implies $c_0\notin S_{i_1\ldots  i_p}(P)$. This proves the lemma.
\end{proof}

Given $\bma\in A_k$, we denote by $\da_i$ the elements of
the set $pr'_\bma(D)\IN E'_\bma$ and by $\Da_i$ the sets ${pr'_\bma}^{-1}(\da_i)\cap D$. Take some $\da_0\in pr'_\bma(D)$ and let $\sa_0=\dfrac{\da_0}{n-1}$ denote the fixed point of the map $S_{\da_0}$. The set ${pr'_\bma}^{-1}(\sa_0)\cap P=\sa_0+P_\bma$ is the section of the cube $P$, and the set $K(\sa_0)={pr'_\bma}^{-1}(\sa_0)\cap K$ is the section of the fractal cube $K$.

\begin{prop}\label{prop:2} 
The section $K(\sa_0)=(\sa_0+P_\bma)\cap K$ is a fractal cube $K_{\Da_0}$ with a digit set $\Da_0={pr'_\bma}^{-1}(\da_0)\cap D$.
\end{prop}

\begin{proof}
For any $d\in \Da_0$, $pr'_\bma(d)=\da_0$, therefore $pr'_\bma(K_{\Da_0})=\sa_0$. Consequently, $K_{\Da_0}\IN K(\sa_0)$.\\

Consider now a homothety $S_\bi$
that is a composition
$S_{i_1 \ldots  i_p}$ of the maps
$S_i(x)=\dfrac{x+d_i}{n}$. By Proposition \ref{c0inK}, the set $pr'_\bma(S_\bi(P^k))$ contains the point $\sa_0$ if and only if for any $d_{i_j}$, $pr'_\bma(d_{i_j})=\da_0$, which implies $K_{\Da_0}\NI K(\sa_0)$.
\end{proof}

\subsection{ The $\al$-sections corresponding to the faces of $P^k$.} If for each $i\in J'_\bma$ the $i$-th entry $\da_{0,i}$ of the vector $\da_0\in E'_\bma$ is equal to 0 or $n-1$, then $\sa_0+P^0_\bma$ is one of the $j$-dimensional faces of $P^k$, parallel to $P_{\bma}$. 

 Each face $P_\bma$, in its turn, is the section $\sa+P^0_\bma$, for which the coefficients $\da_{i}$ of the vector $\da$ are equal to $0$ if $\al_i=-1$ and to $n-1$ if $\al_i=+1$. 
 
 Hence, the relation between $\da$ and $\bma$ for the face $P_\bma$ is expressed by the formula $\da=(n-1)\bma^0$. Summarizing the previous argument, we obtain the following theorem.

\begin{thm}\label{bma0}
For each face $P_\bma$ of the cube $P^k$, the set $K_\bma=K\cap  P_\bma$ is a fractal $k$-cube with a digit set $D_\bma=D\cap (n-1)P_\bma$.
Furthermore, the projection of $K_\bma$ to $E_\bma$ is a fractal cube $K_\bma^0=K_\bma-\bma^0$ with a digit set $D^0_\bma=\pr( D_\bma)= D_\bma-(n-1)\bma^0$.
$\square$\end{thm}

Similarly, if $\bma\bot\bmb$ then $K_{\bma+\bmb}=K_\bma\cap K_\bmb=K\cap P_{\bma+\bmb}$ is a fractal k-cube with a digit set $D_{\bma+\bmb}=D\cap (n-1)P_{\bma+\bmb}$, and its projection to $E_{\bma+\bmb}$ is a fractal cube $K_{\bma+\bmb}^0=K_{\bma+\bmb}-(\bma^0+\bmb^0)$ whose digit set is $D^0_{\bma+\bmb}= D_{\bma+\bmb}-(n-1)({\bma^0+\bmb^0})$.

Applying formulas \eqref{dpk}, \eqref{dpal}, \eqref{dpa0} to the fractal cube $K$, we get the equalities

\begin{equation}\label{kpk}
    \dd K=\bigcup\limits_{\bma\in A\mmm \{0\}}K_\bma; \qquad
\dd K_\bma=\bigcup\limits_{\bmb\in A_\bma\mmm\{0\}}K_{\bma+\bmb};
\mbox{ \quad  and \quad }
    \dd K^0_\bma=\bigcup\limits_{\bmb\in A_\bma\mmm\{0\}}K^0_{\bma+\bmb}+\bmb^0.
\end{equation}

 \subsection{Refinement of the digit set of a fractal cube.}

\begin{dfn}Given a system $\eS=\{S_1,\ldots  ,S_N\}$ of contracting similarities, its {\em $p$-th refinement} is the system $\eS^p$ which consists of all the compositions of length $p$ of the elements of $\eS$, i.e.  $\{S_{i_1}S_{i_2}\ldots  S_{i_p}:(i_1,i_2,\ldots  ,i_p)\in\{1,\ldots  ,N\}^p\}$.\end{dfn}

In the case of a fractal cube, the maps are $S_i(x)=\dfrac{x+d_i}{n}$, where $d_i\in D$. Take some $\bi=i_1i_2\ldots\i_p\in I^p$.
If we calculate the composition $S_\bi=S_{i_1}S_{i_2}\ldots  S_{i_p}$, we obtain
\[S_\bi(x)=\dfrac{x+n^{p-1}d_{i_1}+n^{p-2}d_{i_2}+\ldots  +d_{i_p}}{n^p}.\] 

Then the equality $K=\bigcup\limits_{\bi\in I^p} S_\bi(K)$ becomes $K=\dfrac{D^{(p)}+K}{n^p}$. Here, the set $D^{(p)}$
consists of all elements $\bd_\bi=d_{i_1i_2\ldots  i_p}$, where $(i_1,i_2,\ldots  ,i_p)\in\{1,\ldots  ,N\}^p$ and $\bd_\bi= n^{p-1}d_{i_1}+n^{p-2}d_{i_2}+\ldots  +d_{i_p}$.
 
 Thus, for any $p\in \nn$, the set $K$ may be represented as a fractal cube of order $n^p$ with  the digit set $D^{(p)}$.

We call   $D^{(p)}$ the digit set
 for the  refinement $\eS^p$ of the system $\eS$.\\

 \begin{lem}\label{Dap} For each $\bma\in A_k$, the digit set for $K_\bma$ in $D^{(p)}$ is $(D_\bma)^{(p)}$.\end{lem}
 \begin{proof} By Theorem \ref{bma0}, $(D^{(p)})_\bma=D^{(p)}\cap (n^p-1)P_\bma$.

 Apply the equality $(n^p-1)P_\bma=\sum\limits_{j=0}^{p-1}n^j(n-1)P_\bma$. Since for each $j$, $d_{i_j}\in\{1,2,\ldots  ,n-1\}$, the relation  
\beq(n^{p-1}d_{i_1}+n^{p-2}d_{i_2}+\ldots  +d_{i_p})\in \sum\limits_{j=0}^{p-1}n^j(n-1)P_\bma\eeq is possible if and only if for any $j$, $d_{i_j}\in D_\bma$.
 As a result, for any $\bma\in A_k$, $D^{(p)}_\bma=(D_\bma)^{(p)}$.\end{proof}

\section{Intersection of fractal cubes}

Let $K_1,K_2$ be fractal $k$-cubes of order $n$ with digit sets $D_1,D_2$ and put $F_0=K_1\cap K_2$.\\

To understand the structure of $F_0$ we need to take into account all possible intersections $F_\bma=K_1\cap (K_2+\bma)$, 
where $\bma\in A_k$  and to establish relations between all these sets.\\

\subsection{ The sets $F_\al$ and related digit sets.}

The $\bma$- and $(-\bma)$- faces  of the fractal cubes $K_1$ and $K_2$   are  $K_{1,\bma}=K_1\cap  P_\bma$ and $K_{2,-\bma}=K_2\cap  P_{-\bma}$ respectively. Therefore, $F_\bma$ can be represented as the intersection $K_{1,\bma}\cap (K_{2,-\bma}+\bma)$.\\
  The digit set for the set $K_{1,\bma}$ is $D_{1,\bma}= D_1\cap(n-1)P_\bma$   and the digit set for the set $K_{2,-\bma}$ is $D_{2,-\bma}= D_2\cap(n-1)P_{-\bma}$. Consequently, the fractal cube $(K_{2,-\bma}+\bma)$ has the digit set $D_2\cap(n-1)P_{-\bma}+(n-1)\bma$.

   \begin{prop}\label{falfa}
   Let $K_1,K_2$ be fractal $k$-cubes of order $n$ with digit sets $D_1,D_2$ and let $\bma\in A_k$. The set $F_\bma$ is the intersection of fractal cubes $\hat K_1,\hat K_2$, whose digit sets are $\hat D_1=D_1\cap(n-1)P_\bma$ and $\hat D_{2}=D_2\cap(n-1)P_{-\bma}+(n-1)\bma$ respectively.
   
   Furthermore, for any $\bmg\perp\bma$, the set $F_{\bma+\bmg}$ is the intersection $\hat K_{1}\cap(\hat K_{2}+\bmg)$.
    \end{prop}  
    \begin{proof} The first statement is already proved above. Let us check the second one.
        
    If $\bmb=\bma+\bmg $, then $\bmb\sqsupset\bma$.  The set $F_\bmb$ is the intersection of fractal cubes with digit sets $D_{1,\bmb}=D_1\cap(n-1)P_\bmb$ and $D_{2,-\bmb}=D_2\cap(n-1)P_{-\bmb}+(n-1)\bmb$. Since $\bmb=\bma+\bmg$ and $P_\bmb=P_\bma\cap P_\bmg$, we have $D_{1,\bmb}=\hat D_{1}\cap(n-1)P_\bmg$ and $D_{2,\bmb}=\hat D_{2}\cap(n-1)P_{-\bmg}+(n-1)\bmg$.
   \end{proof}

   \begin{rmk}\label{galfa}  The second statement of Proposition \ref{falfa} shows that the intersection of fractal cubes $\hat K_1$ and $\hat K_2$ and all their faces  may be considered independently of initial sets $K_1, K_2$ as well.\\ The intersection $G_\bma$ of
   the digit sets  $D_1\cap(n-1)P_\bma$ and $D_2\cap(n-1)P_{-\bma}+(n-1)\bma$ is equal to $G_\bma=D_1\cap(D_2+(n-1)\bma)$, and is naturally associated with the set $F_\bma$.
   If $\bma=0$, the set $G_\bma$ becomes $G_0=D_1\cap D_2$.\end{rmk}
 
 \begin{rmk}\label{qbma}
  Let $Q_\bma$ denote a fractal cube  with a digit set $G_\bma$. It satisfies the equations $Q_\bma=\frac{1}{n}(Q_\bma+G_\bma)$, and $\dim_H(Q_\bma)=\log_n\#G_\bma$.   \end{rmk}

  \subsection{The intersection theorem and the graph $\Gamma_\Sa$.}
 
 The following theorem establishes the relations between the sets $F_\bma$:

%

 
 \begin{thm}\label{IFC}
The family $\{F_\bma, \bma\in A_k\}$ of intersections $F_\bma =K_{1}\cap (K_{2}+\bma)$  
 satisfies the system $\Sa$ of equations
\begin{equation}\label{perall}
 F_\bma=\bigcup\limits_{\bmb\sqsupseteq{\bma}}T_{\bma\bmb}(F_\bmb),\qquad \bma\in A_k,\end{equation}
 
 where for any $\bmb\sqsupseteq\bma$, 
 \beq\label{Gab}T_{\bma\bmb}(F_\bmb)=\frac{1}{n}(F_\bmb+G_{\bma\bmb})\mbox{\quad  and \quad}
  G_{\bma\bmb}=D_1\cap(D_2+n\bma-\bmb)\eeq 
\end{thm}

\begin{proof}
 We represent $F_\bma$ as
$K_1\cap (K_2+\bma)=\dfrac{1}{n}\bigl((K_1+D_1)\cap (K_2+D_2+n\bma)\bigr).$ \\

Given $d_1\in D_1$ and $d_2\in D_2$, the intersection $(K_{1}+d_1)\cap (K_{2}+d_2+n\bma)$ is nonempty if $(P+d_1)\cap (P+d_2+n\bma)\neq\0$, which means that the vector $\bmb=d_2-d_1+n\bma$ is an element of $A$. Since, for any coordinate $i=1,\ldots  ,k$, $|(d_2-d_1)_i|\le n-1$, this is possible only if $\bmb\sqsupseteq \bma$.\\

If $\bmb=\bma$, then $d_1=d_2+(n-1)\bma$, consequently $d_1\in D_1\cap(D_2+(n-1)\bma)= G_\bma$.

If $\bmb\sqsupset\bma$, then $(K_{1}+d_1)\cap (K_{2}+d_2+n\bma)=(K_1\cap (K_2+\bmb))+d_1$, and $d_1\in D_1\cap(D_2+n\bma-\bmb)$, which we denote by $G_{\bma\bmb}$. 

Notice that $G_{\bma\bma}=D_1\cap(D_2+n\bma-\bma)$, thereby $G_{\bma\bma}=G_{\bma}$.

As a result, we obtain $F_\bma=\frac{1}{n}\bigcup\limits_{\bmb\sqsupseteq{\bma}}(F_\bmb+G_{\bma\bmb})=
\frac{1}{n}(F_\bma+G_\bma)\cup\bigcup\limits_{\bmb\sqsupset{\bma}}\frac{1}{n}(F_\bmb+G_{\bma\bmb})$.
\end{proof}\bigskip

Theorem \ref{IFC} gives rise to the {\em structure graph $\Ga_\Sa$} of the system
$\Sa$ defined in \eqref{perall}. We use this system and graph as a tool to find the properties of the sets $F_\bma$.

\begin{dfn}The structure graph $\Ga_\Sa$ is a directed graph, in which the vertices are all non-empty sets $F_\bma$ and for each $\bma\sqsubseteq\bmb$ there is an edge $(F_\bma,F_\bmb)$ that is directed from $F_\bma$ to $F_\bmb$ and corresponds to the operator $T_{\bma\bmb}$, if this operator is non-degenerate.\end{dfn}
 In general, the graph $\Ga_\Sa$ would contain $3^k$ vertices and $5^k$ edges, and $3^k$ of these edges are loops from $F_\bma$ to itself.  We mark each edge with $G_{\bma\bmb}$.\\

However, some of the vertices and edges in the graph $\Ga_\Sa$ may vanish. This occurs for those $F_\bma$ that are empty and for those edges $(F_\bma,F_\bmb)$ for which $T_{\bma\bmb}(F_\bmb)=\0$.\\

Note the following obvious statement: 
\begin{equation}\label{Tempty}T_{\bma\bmb}(F_\bmb)=\0\mbox{\quad   if  \quad  }G_{\bma\bmb}=\0\mbox{ \quad   or \quad   }F_\bmb=\0
\end{equation} 

A set $F_\bma$ is empty if $G_\bma=\0$ and for any $\bmb\sqsupset\bma$ the set $F_\bmb+G_{\bma\bmb}=\0$. Applying \eqref{Tempty} to all $\bmb\sqsupset\bma$, we
deduce the following emptiness condition for $F_\bma$:

\begin{lem}
A set $F_\bma=\0$ if and only if for any $\bmb\sqsupseteq\bma$ and for any finite sequence\\ $\bma=\bma_0\sqsubseteq\bma_1\sqsubseteq\ldots \bma_{p-1}\sqsubseteq\bma_p=\bmb$ the product
$\#G_{\bma_0\bma_1}\#G_{\bma_1\bma_2}\ldots  \#G_{\bma_{p-1}\bma_p}\#G_{\bmb}$ is equal to zero. \hfill$\square$\end{lem}

For these reasons, due to the reduction of all empty vertices and empty edges, the structure graph $\Ga$ for the system $\Sa$ defined in Theorem \ref{IFC} has the set of vertices $V_\Sa=\{F_\bma: \bma\in A, F_\bma\neq\0\}$ and the set of edges
$E_\Sa=\{(F_\bma, F_\bmb): \bma\sqsubseteq\bmb,  G_{\bma\bmb}\neq\0, F_\bmb\neq\0\}$. \\

In general, such graph $\Ga_\Sa$ may be disconnected. \\

We say that two vertices $F_\bma, F_\bmb,\bma\sqsubset\bmb $ {\em are connected by a directed path in} $\Ga_\Sa$, if there is a finite sequence $\bma=\bma_0\sqsubset\bma_1\sqsubset\ldots \bma_{p-1}\sqsubset\bma_p=\bmb$ such that for any $j=0,\ldots  ,p$ sets $F_{\bma_j}\neq\0$  and sets $G_{\bma_{j-1}\bma_j}\neq\0$ for $j=1,\ldots  ,p$.\\

 We write $\bmb\succ\bma$ if there is a directed path in $\Ga$ from $F_\bma$ to $F_\bmb$.
 
If $\bmb\succcurlyeq\bma$ or $\bma\succcurlyeq\bmb$ then we say that $\bma$ and $\bmb$ are $\Ga${\em-comparable}.

We denote by $\Ga_\bma$ a subgraph in $\Ga$, whose vertices are all $F_\bmb$ such that $\bmb\succcurlyeq\bma$. We say that
$\bmb$ is {\em maximal} for $\Ga_\bma$, if $\Ga_\bmb$ is a single vertex $F_\bmb$. We say that
$\bmb$ is {\em minimal} for $\Ga_\Sa$, if there is no $\bma$ such that $\bma\prec\bmb$.

It should be noted that, by Proposition \ref{falfa}, the graph $\Ga_\bma$ shows the set of all equations that completely define each of the sets $F_\bmb$, for which $\bmb\succcurlyeq\bma$.\\

Each operator $T_{\bma\bmb}$ that corresponds to some edge in the graph $\Ga_\Sa$ is a non-empty finite union of homotheties, therefore it preserves the dimension, i.e. for any subset $X\IN F_\bmb$, $\dim_H(T_{\bma\bmb}(X))=\dim_H(X)$.\\

\subsection{Example 1. Intersection of two fractal squares consisting of 24 points} Consider the intersection of fractal squares $K_1$ and $K_2$ of order 6 with digit sets $D_1$ and $D_2$. In the figure below left, we represent the digit sets by the sets of boxes $T_1(P)=\dfrac{D_1+P}{6}$ (red) and $T_2(P)=\dfrac{D_2+P}{6}$ (blue). On the right, one can see the fractal squares $K_1$ and $K_2$. \\

Most of the sets $G_\bma$, namely, $G_0,G_{(1,0)},G_{(-1,0)},G_{(0,1)},G_{(0,-1)},G_{(1,1)},G_{(1,-1)},G_{(-1,1)}$ are empty and only $G_{(-1,-1)}=(0,0)$. Consequently, 
$F_{(1,1)}=F_{(1,-1)}=F_{(-1,1)}=\0$ and $F_{(-1,-1)}=\{(0,0)\}$.\\

\begin{figure}[H]
    \centering
    \resizebox{0.35\textwidth}{!}{\begin{tikzpicture}
    \fill[red!75]\foreach \a in
        {(0,0), (2,0), (4,0), (2,1), (4,1), (0,2), (1,2), (3,2), (2,3), (4,3), (0,4), (1,4), (3,4)} {\a rectangle +(1,1)};
    \fill[blue!75]\foreach \a in 
        {(1,1), (2,2), (3,3), (4,4), (5,5), (3,1), (5,1), (4,2), (1,3), (5,3), (2,4), (1,5), (3,5)} {\a rectangle +(1,1)};
    \draw[gray, line width=0.6mm] (0,0) grid (6,6);
    \draw[black, line width=0.6mm] (0,0) rectangle (6,6);
\end{tikzpicture}}\hspace{0.15\textwidth}
\includegraphics[width=0.35\textwidth]{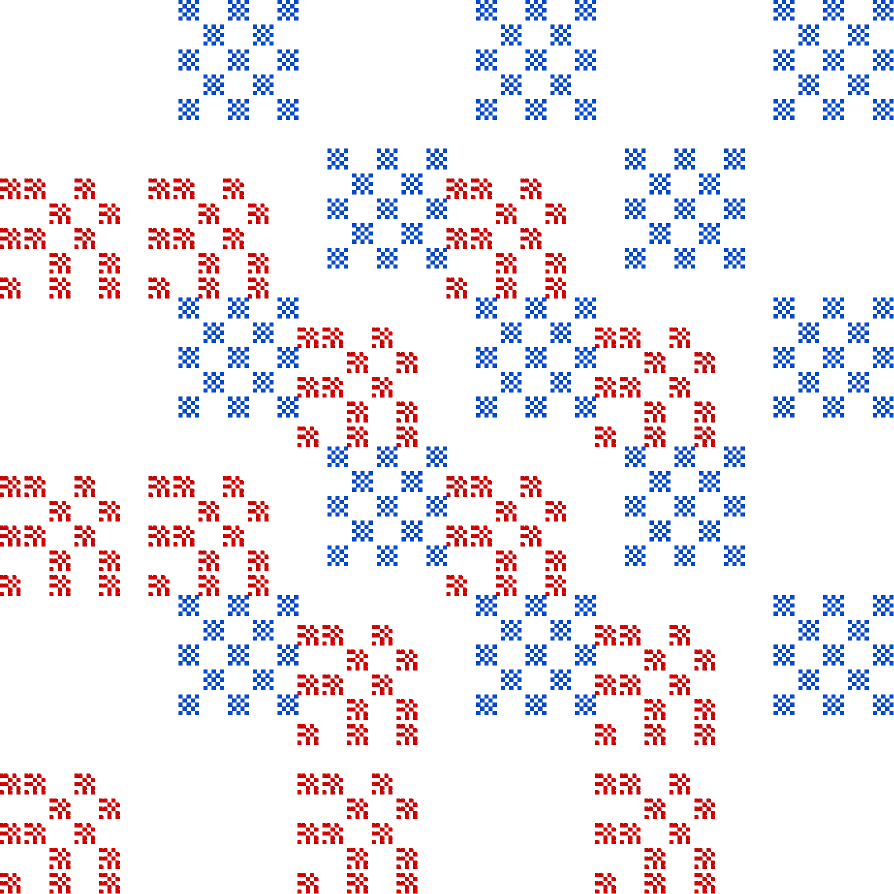}
     \caption{The digit set diagram for $D_1$ and $D_2$ (left) and the fractal squares $K_1$ and $K_2$ (right).}
    
\end{figure}
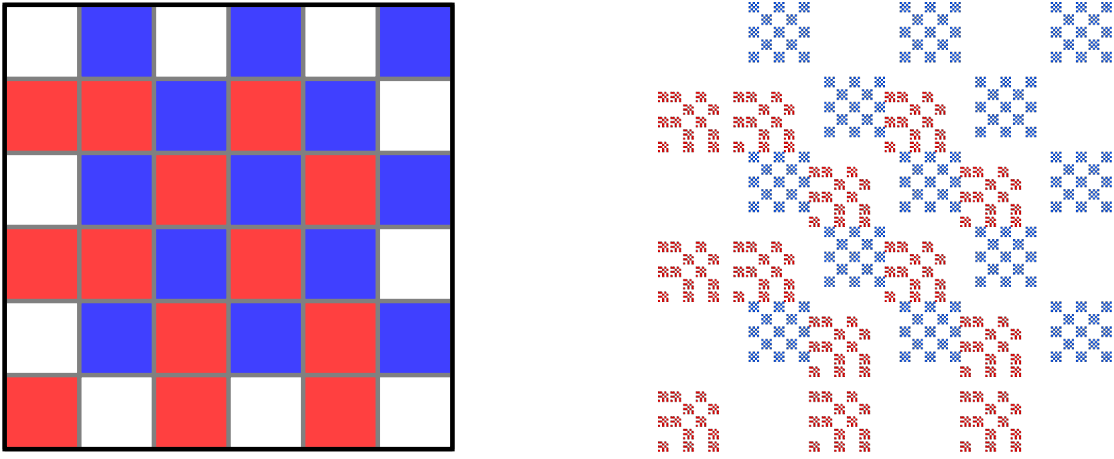

The set $F_{(1,0)}$ is empty because $G_{(1,0)},F_{(1,1)}$ and $F_{(1,-1)}$ are empty. For the same reason, $F_{(0,1)}=\0$. \\

The set $G_{(-1,0)(-1,-1)}=\{(0,2),(0,4)\}$ and $G_{0(-1,0)}=\{(1,2),(1,4),(2,3),(3,2),(3,4),(4,3)\}$.
         The sets $G_{(0,-1)(-1,-1)}$ and $G_{0(0,-1)}$ are obtained from the previous two sets by toggling $x$ and $y$.\\

Thus, after the reduction of empty vertices and edges, 
the graph $\Ga_\Sa$ contains four vertices $F_0,F_{(-1,0)},F_{(0,-1)},F_{(-1,-1)}$ and four edges, corresponding to $G_{(-1,0)(-1,-1)}, G_{(0,-1)(-1,-1)},G_{0(-1,0)}$ and $G_{0(0,-1)}$.

The calculation using formula \eqref{perall} shows that
$\#F_{(-1,0)}=\#F_{(0,-1)}=2$ and $\#F_0=2 \#G_{0(-1,0)}+2\#G_{0(0,-1)}=24.$

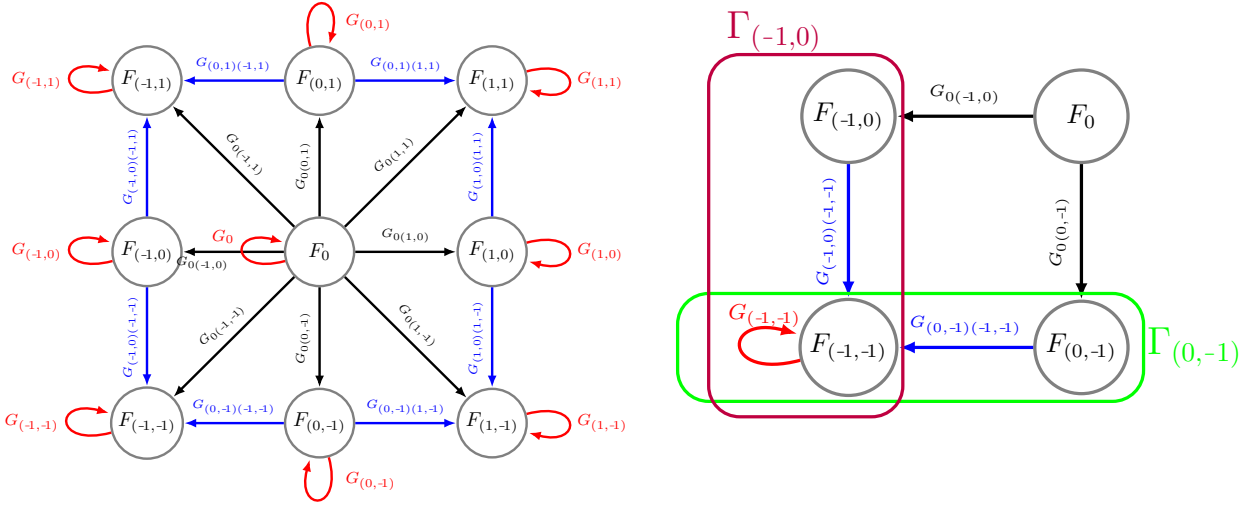
\begin{figure}[h!]
    \centering
\resizebox{.5\textwidth}{!}{
\begin{tikzpicture}
    \node[gdst, shift={(0,0)}] (f1){\footnotesize$F_{(\sm\!1,\sm\!1)}$};
    \node[gdst, shift={(3,0)}] (f2){\footnotesize$F_{(0,\sm\!1)}$};
    \node[gdst, shift={(6,0)}] (f3){\footnotesize$F_{(1,\sm\!1)}$};
    \node[gdst, shift={(0,3)}] (f4){\footnotesize$F_{(\sm\!1,0)}$};
    \node[gdst, shift={(3,3)}] (f5){\footnotesize$F_{0}$};
    \node[gdst, shift={(6,3)}] (f6){\footnotesize$F_{(1,0)}$};
    \node[gdst, shift={(0,6)}] (f7){\footnotesize$F_{(\sm\!1,1)}$};
    \node[gdst, shift={(3,6)}] (f8){\footnotesize$F_{(0,1)}$};
    \node[gdst, shift={(6,6)}] (f9){\footnotesize$F_{(1,1)}$};
    \path[<-,>={Latex[length=6pt]}, very thick] 
        (f1) edge[blue] node[above] {\tiny$G_{(0,\sm\!1)(\sm\!1,\sm\!1)}$} (f2)
             edge[blue] node[above, rotate=90] {\tiny$G_{(\sm\!1,0)(\sm\!1,\sm\!1)}$} (f4)
             edge node[above, rotate=45] {\tiny$G_{0(\sm\!1,\sm\!1)}$} (f5)
             edge[loop left, red] node[left] {\scriptsize$G_{(\sm\!1,\sm\!1)}$} (f1)
        (f3) edge[blue] node[above] {\tiny$G_{(0,\sm\!1)(1,\sm\!1)}$} (f2)
             edge[blue] node[above, rotate=90] {\tiny$G_{(1,0)(1,\sm\!1)}$}(f6)
             edge node[above, rotate=-45] {\tiny$G_{0(1,\sm\!1)}$} (f5)
             edge[loop right, red] node[right] {\scriptsize$G_{(1,\sm\!1)}$} (f3)
        (f7) edge[blue] node[above, rotate=90] {\tiny$G_{(\sm\!1,0)(\sm\!1,1)}$} (f4)
             edge node[above, rotate=-45] {\tiny$G_{0(\sm\!1,1)}$} (f5)
             edge[blue] node[above] {\tiny$G_{(0,1)(\sm\!1,1)}$} (f8)
             edge[loop left, red] node[left] {\scriptsize$G_{(\sm\!1,1)}$} (f7)
        (f9) edge[blue] node[above] {\tiny$G_{(0,1)(1,1)}$} (f8)
             edge node[above, rotate=45] {\tiny$G_{0(1,1)}$} (f5)
             edge[blue] node[above, rotate=90] {\tiny$G_{(1,0)(1,1)}$} (f6)
             edge[loop right, red] node[right] {\scriptsize$G_{(1,1)}$} (f9)
        (f2) edge node[above, rotate=90] {\tiny$G_{0(0,\sm\!1)}$} (f5)
             edge[loop below, red] node[above right] {\scriptsize$\quad G_{(0,\sm\!1)}$} (f2)
        (f4) edge node[below left=-0.8mm] {\tiny$G_{0(\sm\!1,0)}$} (f5)
             edge[loop left, red] node[left] {\scriptsize$G_{(\sm\!1,0)}$} (f4)
        (f6) edge node[above] {\tiny$G_{0(1,0)}$} (f5)
             edge[loop right, red] node[right] {\scriptsize$G_{(1,0)}$} (f6)
        (f8) edge node[above, rotate=90] {\tiny$G_{0(0,1)}$} (f5)
             edge[loop above, red] node[below right] {\scriptsize$\quad G_{(0,1)}$}(f8)
        (f5) edge[loop left, red] node[above left] {\scriptsize$G_{0}$} (f5);
\end{tikzpicture}}\hspace{.02\textwidth}
\resizebox{0.45\textwidth}{!}{\raisebox{3em}
{\begin{tikzpicture}
    \node[gdst, shift={(0,0)}] (f1){\footnotesize$F_{(\sm\!1,\sm\!1)}$};
    \node[gdst, shift={(3,0)}] (f2){\footnotesize$F_{(0,\sm\!1)}$};
    \node[gdst, shift={(0,3)}] (f4){\footnotesize$F_{(\sm\!1,0)}$};
    \node[gdst, shift={(3,3)}] (f5){\footnotesize$F_{0}$};
    \path[<-,>={Latex[length=6pt]}, very thick] 
        (f1) edge[blue] node[above] {\tiny$G_{(0,\sm\!1)(\sm\!1,\sm\!1)}$} (f2) 
             edge[blue] node[above, rotate=90] {\tiny$G_{(\sm\!1,0)(\sm\!1,\sm\!1)}$} (f4) 
             edge[loop left, red] node[above=0.1cm] {\scriptsize$\qquad G_{(\sm\!1,\sm\!1)}$} (f1)
        (f2) edge node[above, rotate=90] {\tiny$G_{0(0,\sm\!1)}$} (f5)
        (f4) edge node[above] {\tiny$G_{0(\sm\!1,0)}$} (f5);
    \draw[green, very thick, rounded corners=4mm] (-2.2,-0.7)rectangle(3.8,0.7);
    \draw[purple, very thick, rounded corners=4mm] (-1.8,-0.9)rectangle(0.7,3.8);
    \node[shift={(3.7,0)}, right, green]{\large$\Gamma_{(0,\sm1)}$};
    \node[shift={(-.2,4.1)}, left, purple]{\large$\Gamma_{(\sm1,0)}$};
\end{tikzpicture}}}

\caption{The maximal possible structure graph $\Ga_\Sa$ for the intersection of two fractal squares (on the left) and the graph $\Ga_\Sa$ for Example 1 (on the right). The picture on the right shows the subgraphs $\Ga_{(-1,0)}$ and $\Ga_{(0,-1)}$. }
\end{figure} 

\subsection{Intersections of the pieces of a fractal cube $K$.}

In the case when  $K_1=K_2=K$, the sets $F_\bma$ are the intersections of opposite faces $K_\bma$ and $K_{-\bma}$ of the same fractal cube $K$. If $\bma=0$, then $F_0=K$, otherwise $F_\bma=F_{-\bma}$ for any $\bma\neq 0$. Direct computation shows that
for any $\bma$, $D_\bma=D\cap(D+(n-1)\bma)$ implies that $D_{-\bma}=D_\bma-(n-1)\bma$. Similarly,
$D_{-\bma-\bmb}=D\cap(D-n\bma+\bmb)=D_{\bma\bmb}-n\bma+\bmb$.

\subsection{Finiteness condition for the measure $H^s(F_\al)$.}
\begin{dfn}\label{nubma}
 For any $\bma\in A_k$ we put  $\nu(\bma)=\max\{\#G_\bmb : \bmb\succcurlyeq\bma\}$ and $s(\bma)=\log_n\nu(\bma)$.   
\end{dfn}
On account of Remark \ref{qbma}, we may also write
$s(\bma)=\max\limits_{\bmb\succcurlyeq\bma}\{\dim_HQ_\bmb  \}$.
\begin{thm}\label{dimthm}
For any $\bma\in A_k$, the Hausdorff dimension  of the set $F_\bma$ is $s(\bma)$, and its $s(\bma)$-measure $H^{s(\bma)}(F_\bma)$ is positive.\\ The measure
$H^{s(\bma)}(F_\bma)$ is finite if and only if the set $\{\bmb\succcurlyeq\bma : \#G_\bmb=\nu(\bma)\}$ does not contain $\Ga$-comparable elements.
\end{thm}

{\bf Proof.} If we put $B=\bigcup\limits_{\bmb\succ\bma}T_{\bma\bmb}(F_\bmb)$, the equation \eqref{perall} becomes 
\beq F_\bma=T_{\bma\bma}(F_\bma)\cup B\eeq Iterating the last equation and keeping in mind that $\lim\limits_{m\to\8}T^m_{\bma\bma}(F_\bma)=Q_\bma$, we obtain that \beq\label{FviaB} F_\bma=Q_\bma\cup \bigcup\limits_{m=0}^\8T_{\bma\bma}^m(B)\eeq

This equality implies that $\dim_HF_\bma=\max(\dim_HQ_\bma,\dim_HB)$.
All non-degenerate operators $T_{\bma\bmb}$ preserve the dimension, that is, $\dim_H T_{\bma\bmb}(F_\bmb)=\dim_H F_\bmb$. It follows that
$\dim_HB=\max\limits_{\bmb\succ\bma}\dim_HF_\bmb$, hence
$\dim_HF_\bma=\max(\dim_HQ_\bma,\max\limits_{\bmb\succ\bma}\dim_HF_\bmb)$.

Suppose now that $\dim_HF_\bma=s'>s(\bma)$. Then $\dim_HB=s'$, hence
there is $\bmb\succ\bma$ such that $\dim_HF_\bmb=s'$.
Applying this argument to the set $F_\bmb$, we see that the set of all $\bmb\succ\bma$, such that $\dim_HF_\bmb=s'$, cannot contain a minimal element. This is impossible because it is finite. 

Consequently, $\dim_HF_\bma=s(\bma)$. \\

If $\bmb\succ\bma$ and $\#G_\bmb=\nu(\bma)$, then
$H^{s(\bma)}(Q_\bmb)$ is positive,   hence $H^{s(\bma)}(F_\bma)$ is also positive.

If $\#G_\bma>\#G_\bmb$ for any $\bmb\succ\bma$, then $H^s(Q_\bmb)=0$ for all $\bmb\succ\bma$, therefore, $H^s(F_\bma)=H^s(Q_\bma)$ is finite and positive.\\
This completes the proof of the first statement of Theorem \ref{dimthm}.\\

Suppose now $\#G_\bma<\nu(\bma)$ and $H^s(B)=h>0$. Then $H^s(T^m_\bma(B))=\dfrac{(\#G_\bma)^m h}{\nu(\bma)^m}$. Taking a sum over all $m$ from 0 to infinity, we obtain
$H^s(F_\bma)\leq \dfrac{h\nu(\bma)}{\nu(\bma)-\#G_\bma}$, hence $H^s(F_\bma)$ is positive and finite.\\

Now we show that if for some $\bmb\succ\bma$, $\#G_\bmb=\#G_\bma=\nu(\al)$, then
$H^s(F_\bma)$ is infinite. This will be proved in the following lemma, which finalizes the proof of the theorem.\\

 On account of Theorem\ref{bma0},   it      sufficient to consider the situation where $\bma=0$.\\

\begin{lem}
If $\#G_0=\#G_\bma=l$, $\log_nl=s$ and for any $\bmb\succ 0$, $\#G_\bmb\le l$. Then
$H^s(F_0)$ is infinite.
\end{lem}
\begin{proof}

 {\bf Case 1}: There is
$d_1\in D_1$ such that $d_1+\bma\in D_2$.\\

 Consider fractal cubes $Q_0$ and $Q_\bma$. Note that both sets $Q_0$ and $Q_\bma$ have finite positive measure in dimension $s$. 
 
 Consider a set $Q_*$ defined by equation $ Q_*= \dfrac{Q_*+G_0}n\bigcup\dfrac{d_1+Q_\bma}n$. Define the operator $T_0$ by the equality $T_0(A)= \dfrac{A+G_0}n$. Let $B=\dfrac{d_1+Q_\bma}n$.
 
 Let $\hT$ be the operator defined by the equality $\hT(A)= \dfrac {I^k+A}n$ whose attractor is the whole cube $P$.
 
  Consider the sets $P_\bma^{(m)}=\hT^m(P_\bma)$. Inclusion $P_\bma\IN\hT(P_\bma) $ implies that these sets form an increasing nested sequence $P_\bma\IN P_\bma^{(1)}\IN P_\bma^{(2)}\ldots$.
  
   Note that $B\IN P_\bma^{(1)}\mmm P_\bma$ and $T_0^m(B)\IN \hT^m(B)\IN P_\bma^{(m+1)}\mmm P_\bma^{(m)}$. The sets $P_\bma^{(m+1)}\mmm P_\bma^{(m)}$ are disjoint, and thus the sets $T_0^m(B)$ are disjoint, too.
 
On the other hand, since $\#G_0=\#G_\bma=l$, $H^s(T_0^m(B))=H^s(B)$. Consequently, the set $\bigcup\limits_{m=0}^\8 T_0^m(B)$ has an infinite Hausdorff measure in dimension $s$.\\
 
 {\bf Case 2.}   There is a sequence $0\prec\bma_1\prec\ldots\prec\bma_{p}$ such that $\bma_p=\bma$ and\\
  $\#G_{0\bma_1}\#G_{\bma_1\bma_2}\ldots\#G_{\bma_{p-1}\bma_p}\neq 0$.
  Then for each $i\in\{1,\ldots,p\}$ there are $d_i\in D_1$ and $\da_i\in D_2$ such that $\da_i=d_i-n\bma_{i-1}+\bma_i$.
 Then \[\sum\limits_{i=1}^p n^{p-i}\da_i=\sum\limits_{i=1}^p n^{p-i}(d_i-n\bma_{i-1}+\bma_i)=\sum\limits_{i=1}^p n^{p-i} d_i +\bma_p\] That is, $\bm{\da}_{i_1\ldots i_p}-\bm{d}_{i_1\ldots  i_p}=\bma_p$.
 This means that if we consider $K_1$ and $K_2$ as fractal cubes of the order $n^p$ with digit sets $D_1^{(p)}$ and $D_2^{(p)}$, respectively, then the set $G^{(p)}_{0\bma}$ is nonempty. 
 At the same time, the equalities $D^{(p)}_{1,\bma}= D^{(p)}_1\cap(n^p-1)P_\bma=(D_{1,\bma})^{(p)}$ and $D^{(p)}_{2,-\bma}= D^{(p)}_2\cap(n^p-1)P_{-\bma}=(D_{2,-\bma})^{(p)}$ imply that $G^{(p)}_\bma=(G_\bma)^{(p)}$. Then it follows that $\#G^{(p)}_\bma=l^p$ and $\log_{n^p}\#G^{(p)}=s$. Hence, the fractal cube of order $n^p$ with the digit set  $G^{(p)}_\bma$ coincides with $Q_\bma$.
  
  This situation is the same as in Case 1, therefore, the measure $H^s(F_0)$ is infinite.
 \end{proof}
 
 \subsection{The cardinality of the set $F_\al$.}
\begin{thm}
1.The set $F_\bma$ is a singleton if  $\Gamma_\bma$ is a chain $\bma=\bma_1\prec\ldots\prec\bma_p$ in which  for all $j\le p-1$,
$\# G_{\bma_j\bma_{j+1}}=1$, $G_{\bma_j}=\0$ and $\#G_{\bma_p}=1$.\\
2. The set $F_\bma$ is finite if for all maximal vertices $\bmb$ in $\Gamma_\bma$, $\#G_\bmb=1$ and $G_{\bmb}=\0$ for all other vertices in $\Gamma_\bma$.
In this case $\#F_\bma$ is equal to the sum of all compositions $\prod \limits_{j=1}^{p-1}\# G_{\bma_j\bma_{j+1}}$ taken over all chains $\bma=\bma_1\prec\ldots\prec\bma_p=\bmb$, where $\bmb$ is maximal in $\Gamma_\bma$;\\
3. The set $F_\bma$ is countable if for all vertices $\bmb$ in $\Gamma_\bma$, $\#G_\bmb\le 1$.\\
4. The set $F_\bma$ is uncountable if there is a vertex $\bmb$ in $\Gamma_\bma$, such that $\#G_\bmb> 1$.
\end{thm}

\begin{cor}A fractal cube $K$ possesses one-point intersection  property iff the structure graph $\Gamma(\Sa)$ is a union of chains $0\prec\bma_{i1}\prec\ldots\prec\bma_{ip_i}$ for which all $\bma_{ij}$ are different and such that for all $i$ $\#G_{\bma_{ip_i}}=1$ and for all $i,j$ for which $j\le p_i-1$,
$\# G_{\bma_{ij}\bma_{i,j+1}}=1$, $G_{\bma_{ij}}=\0$.

A fractal cube $K$ has a finite intersection property iff for all maximal $\bma$ in the graph $\Ga(\Sa)$, $\#G_\bma=1$ and for all other $\bma\neq 0$, $\#G_\bma=0$.  

\end{cor}

\end{document}